 \documentclass[11pt, reqno]{amsart}
   \usepackage{amsmath}
   \usepackage{amsthm}
   \usepackage{amsfonts}
   \usepackage{amssymb}
   \usepackage[a4paper, margin=0.9in]{geometry}
   \usepackage{enumerate}
     \usepackage{color}
     \usepackage[all]{xy}
     \usepackage{amssymb}
\usepackage{amscd}
\usepackage{graphicx}
\usepackage{braket}
\usepackage{tikz-cd}
\usepackage{hyperref}
\usepackage{mathtools}

\tikzset{
    slanted/.style={rotate=-90, anchor = south},
    slantedswap/.style={rotate=-90, anchor=north}
}

\usepackage[utf8]{inputenc}

\usepackage[intoc, refpage]{nomencl}
  \newcommand{\twodigitspage}{
   \ifnum \value{page} < 10
   0
   \fi
   \thepage
  }
   \theoremstyle{definition}
   \newtheorem{definition}{Definition}[section]
   \newtheorem{remark}[definition]{Remark}
   
   \newtheorem{examples}[definition]{Examples}

   \theoremstyle{plain}   
   \newtheorem{proposition}[definition]{Proposition}
   \newtheorem{lemma}[definition]{Lemma}
   \newtheorem{theorem}[definition]{Theorem}
   \newtheorem{corollary}[definition]{Corollary}

   \newtheorem{hypothesis}[definition]{Hypothesis}

   \theoremstyle{definition}

   \newcommand{\coker}{\operatorname{coker}}

\newcommand{\Hom}{\mathrm{Hom}}

\newcommand{\Gal}{\operatorname{Gal}}

\newcommand{\Spec}{\operatorname{Spec}}

\newcommand{\catname}[1]{\textnormal{{\textsf{#1}}}}

\newcommand{\Der}{\catname{D}}
\newcommand{\R}{\catname{R}}
\newcommand{\Le}{\catname{L}}

\newcommand{\Tor}{\catname{Tor}}

\newcommand{\tor}{\mathrm{tor}}
\newcommand{\tf}{\mathrm{tf}}

\newcommand{\p}{\mathfrak{p}}

\newcommand{\pnull}{\mathrm{null}}

 \makeatletter
\newcommand{\colim@}[2]{%
  \vtop{\m@th\ialign{##\cr
    \hfil$#1\operator@font colim$\hfil\cr
    \noalign{\nointerlineskip\kern1.5\ex@}#2\cr
    \noalign{\nointerlineskip\kern-\ex@}\cr}}%
}
\newcommand{\colim}{%
  \mathop{\mathpalette\colim@{\rightarrowfill@\scriptscriptstyle}}\nmlimits@
}
\renewcommand{\varprojlim}{%
  \mathop{\mathpalette\varlim@{\leftarrowfill@\scriptscriptstyle}}\nmlimits@
}
\renewcommand{\varinjlim}{%
  \mathop{\mathpalette\varlim@{\rightarrowfill@\scriptscriptstyle}}\nmlimits@
}
\makeatother

\makeatletter
\newcommand{\Rlim@}[2]{
  \vtop{\m@th\ialign{##\cr
    \hfil$#1\catname{R}\operator@font lim$\hfil\cr
    \noalign{\nointerlineskip\kern1.5\ex@}#2\cr
    \noalign{\nointerlineskip\kern-\ex@}\cr}}%
}
\newcommand{\Rlim}{%
  \mathop{\mathpalette\Rlim@{\leftarrowfill@\scriptscriptstyle}}\nmlimits@
}
\makeatother

  \newcommand*{\id}{\mathrm{id}}

  \DeclareMathOperator*{\im}{im}

\newcommand{\QQ}{\mathbb{Q}}
\newcommand{\ZZ}{\mathbb{Z}}

\newcommand{\NN}{\mathbb{N}}

\newcommand{\cP}{\mathcal{P}}

\newcommand{\fq}{\mathfrak{q}}
\newcommand{\fp}{\mathfrak{p}}

\makeatletter
\def\th@plain{%
  \thm@notefont{}
  \itshape 
}
\def\th@definition{%
  \thm@notefont{}
  \normalfont 
}
\makeatother

\begin{document}

\title[On non-Noetherian Iwasawa theory]{On non-Noetherian \\
 Iwasawa theory}

 \author{David Burns, Alexandre Daoud and Dingli Liang}

\begin{abstract} We prove a general structure theorem for finitely presented torsion modules over a class of commutative rings that need not be Noetherian. As a first application, we then use this result to study the  Weil-\'etale cohomology groups of $\mathbb{G}_m$ for curves over finite fields. 
\end{abstract}

\address{King's College London,
Department of Mathematics,
London WC2R 2LS,
U.K.}
\email{david.burns@kcl.ac.uk, alexandre.daoud@kcl.ac.uk, dingli.1.liang@kcl.ac.uk}


\vspace*{-1cm}

\maketitle


\section{Introduction}

Let $p$ be a prime, $k$ the function field of a smooth projective curve over the field with $p$ elements and $K$ a Galois extension of $k$ for which $\Gal(K/k)$ is topologically isomorphic to the direct product $\ZZ_p^\NN$ of a countably infinite number of copies of $\ZZ_p$. Then the completed $p$-adic group ring $\ZZ_p[[\ZZ_p^\NN]]$ is not Noetherian and so classical techniques of Iwasawa theory do not apply in this setting. With this problem in mind, Bandini, Bars and  Longhi introduced a notion of `pro-characteristic ideal' as a generalisation of the classical notion of characteristic ideal, and used it to study several natural Iwasawa-theoretic modules over $K/k$ (cf. \cite{bbl, bbl2, bbl3}). These efforts culminated in their proof, with Anglès, of a main conjecture for divisor class groups over Carlitz-Hayes cyclotomic extensions of $k$ (see \cite{abbl}) and, more recently, both Bandini and Coscelli \cite{bc} and Bley and Popescu \cite{bp} have extended this sort of result to a wider class of Drinfeld modular towers. 

By adopting a slightly more conceptual algebraic approach, we shall now strengthen the theory developed in these earlier articles. As the starting point for this, we identify a natural class of  commutative rings (that includes, as a special case, all rings of the form $\ZZ_p[[\ZZ_p^\NN \times G]]$ with $G$ a finite abelian group) that are not, in general, Noetherian, but for which one can prove a structure theorem for the category of finitely presented torsion modules (see Theorem \ref{bourbaki-theorem}). This result is perhaps of some independent interest and, in particular, leads naturally to a general notion of characteristic ideal that extends and  refines the pro-characteristic ideal construction used previously. 

We next prove that the inverse limits with respect to corestriction of the $p$-completions of the degree one Weil-\'etale cohomology groups of $\mathbb{G}_m$ over finite extensions of $k$ in $K$ are finitely presented torsion $\ZZ_p[[\ZZ_p^\NN]]$-modules. By applying our structure theory to these modules, we are then able to derive stronger, and more general,  versions of the main results of each of \cite{abbl}, \cite{bc} and \cite{bp} (see Theorem \ref{main result} and Remarks \ref{abbl rem} and \ref{bp rem}). At the same time, this approach also allows us to prove that, surprisingly, the inverse limit with respect to norms of the $p$-parts of the degree zero divisor class groups of finite extensions of $k$ in $K$ is finitely generated as a $\ZZ_p[[\ZZ_p^\NN]]$-module for only a remarkably small class of extensions $K/k$ (see Corollary \ref{fg cor}). 

There are also natural families of Galois extensions of group $\ZZ_p^\NN$ in number field settings (see, for example, the `cyclotomic radical $p$-extensions' described by Min\'a\v c et al in \cite[Th. A.1]{mrt}). The algebraic results obtained here can be used in a similar way to study Iwasawa-theoretic modules over such extensions, and these applications will be discussed elsewhere.    
\smallskip

\noindent{}{\bf Acknowledgements} The authors are very grateful to Francesc Bars and Meng Fai Lim for helpful advice about the literature in this area, and to Andrea Bandini, Werner Bley and Cristian Popescu for their interest and generous encouragement. 

\section{Structure theories over non-Noetherian rings}\label{structure section}


In this section we fix a commutative unital ring $A$ and write $Q(A)$ for its total quotient ring and $\cP = \cP_A$ for the set of its prime ideals of height one. Given an $A$-module $M$ we write $M_\tor = M_{A\cdot{\rm tor}}$ for the $A$-torsion submodule of $M$ and $M_\tf$ for the quotient of $M$ by $M_\tor$. We then define a (possibly empty) subset of $\cP$ by setting   
\[ \cP(M) = \cP_A(M) :=  \cP\cap \mathrm{Support}(M_\tor)  = \{\fp \in \cP: (M_\tor)_\fp \not= (0)\}.\]

\subsection{Finitely presented modules}

The following notion will play a key role in the sequel. 

\begin{definition}\label{tame-definition}
    A finitely generated $A$-module $M$ will be said to be \textit{amenable} if it has both of the following properties: 
\begin{itemize}
\item[(P$_1$)] for every prime ideal $\fp$ that is maximal amongst those contained in $\bigcup_{\mathfrak{q} \in \cP(M)} \mathfrak{q}$, the localisation $A_\fp$ is a valuation ring (that is, its ideals are linearly ordered by inclusion). 
\item[(P$_2$)] $\cP(M)$ is finite and every prime ideal in $\cP(M)$ is finitely generated. 
\end{itemize}
\end{definition}

\begin{remark}\label{prime-avoidance-remark}
    If $\cP(M)$ is finite (as is required by (P$_2$) and automatically the case if, for example, $A$ is Noetherian), then the prime avoidance lemma implies that (P$_1$) is satisfied if and only if $A_\mathfrak{q}$ is a valuation ring for all $\mathfrak{q}$ in $\cP(M)$. In the general case, however, prime ideals that are contained in $\bigcup_{\mathfrak{q} \in \cP(M)} \mathfrak{q}$ need not have height one.
\end{remark}

As usual, an $A$-module is said to be \textit{pseudo-null} if its localisation  vanishes at every prime ideal of height at most one,  and a map of $A$-modules is said to be a \textit{pseudo-isomorphism} if its kernel and cokernel are both pseudo-null.

We can now prove the structure result that is the starting point of our theory. 

\begin{theorem}\label{bourbaki-theorem} Let $M$ be a finitely presented $A$-module with property (P$_1$). Then the following claims are valid. 

\begin{itemize}
\item[(i)] If $M$ is torsion, then there exists an $A$-module $N$, a finite family of principal ideals $\set{L_\tau}_{\tau\in \mathcal{T}}$ and a pseudo-isomorphism of $A$-modules  
    \begin{align}
        M \oplus N \to \bigoplus_{\tau \in \mathcal{T}} A/L_\tau.\label{pseudiosmorphism-display}
    \end{align}
\item[(ii)] If $Q(A)$ is semisimple, then the following claims are also valid.
\begin{itemize}
\item[(a)] If 
$M$ is both amenable and torsion, then in the pseudo-isomorphism (\ref{pseudiosmorphism-display}) one can take $N$ to be zero and replace each $L_\tau$ by a power of a prime ideal in $\cP(M)$.
\item[(b)] In general, there exists a pseudo-isomorphism of $A$-modules $M \to M_\tor \oplus M_\tf$. 
\end{itemize}
\end{itemize}  
\end{theorem}

\begin{proof} To prove claim (i) we assume that $M$ is $A$-torsion. We also note  that if $\cP(M) = \emptyset$, then $M$ is pseudo-null and there is nothing to prove. We therefore assume that $\cP(M)\not= \emptyset$, set $S := A\setminus \bigcup_{\fp \in \cP(M)} \fp$ and write $(-)'$ for the localization functor $S^{-1}(-)$. 

The maximal ideals of $A'$ are in one-to-one correspondence with the primes of $A$ that are maximal amongst those contained in $\bigcup_{\fq \in \cP(M)} \fq$. Hence, from condition (P$_1$), it follows that the localisation of $A'$ at each maximal ideal is a valuation ring. We may therefore apply Warfield's Structure Theorem \cite[Th. 3]{warfield} to deduce the existence of an $A'$-module $N'$ and a finite collection $\{a_\tau'\}_{\tau \in \mathcal{T}}$ of elements of $A'\setminus (A')^\times$ for which there is an isomorphism of $A'$-modules
    \begin{align}
        \psi: M' \oplus N' \cong \bigoplus_{\tau \in \mathcal{T}} A'/(a_\tau'). \label{warfield-structure}
    \end{align}
We now choose elements $\{a_\tau\}_{\tau \in \mathcal{T}}$ of $A\setminus S = \bigcup_{\fp \in \cP(M)} \fp$ with $(a_\tau)' = (a_\tau')$ for each $\tau\in \mathcal{T}$. Then, since both $M$ and $\bigoplus_{\tau \in \mathcal{T}} A/(a_\tau)$ are finitely presented $A$-modules (the former by assumption), the canonical maps  

    \begin{align}\label{hom-commutes-with-localisation}
        \Hom_{A}\bigl(M, \bigoplus_{\tau \in \mathcal{T}} A/(a_\tau)\bigr)' \xrightarrow{\sim}&\, \Hom_{A'}\bigl(M', \bigoplus_{\tau \in \mathcal{T}} A'/(a_\tau')\bigr),\\
        \Hom_{A}\bigl(\bigoplus_{\tau \in \mathcal{T}} A/(a_\tau),M\bigr)' \xrightarrow{\sim}&\, \Hom_{A'}\bigl(\bigoplus_{\tau \in \mathcal{T}} A'/(a_\tau'),M'\bigr)  \notag\end{align}
   are both bijective. This implies the existence of homomorphisms of $A$-modules 
   
   \[ \iota_1: M \to \bigoplus_{\tau \in \mathcal{T}} A/(a_\tau)\quad\text{and}\quad \iota_2: \bigoplus_{\tau\in \mathcal{T}} A/(a_\tau) \to M\]
   such that, for suitable elements $s_1$ and $s_2$ of $S$, the maps $\iota'_1/s_1$ and $\iota'_2/s_2$ are respectively equal to the composites
   \[ M' \xrightarrow{({\rm id},0)} M'\oplus N' \xrightarrow{\psi} \bigoplus_{\tau\in \mathcal{T}} A'/(a_\tau')\quad\text{and}\quad \bigoplus_{\tau\in \mathcal{T}}A'/(a_\tau') \xrightarrow{\psi^{-1}} M'\oplus N' \xrightarrow{({\rm id},0)} M'.\]
Set $N := \ker(\iota_2)$. Then, since the endomorphism $\iota_2\circ \iota_1$ of $M$ is given by multiplication by $s_2s_1$ and the latter element is not contained in any prime in $\cP(M)$, the modules $\ker(\iota_1)$, ${\rm coker}(\iota_2)$ and $\iota_1(M)\cap N$ are all pseudo-null and the inclusion 
\[ \iota_1(M) + N \to \bigoplus_{\tau\in \mathcal{T}} A/(a_\tau)\]
is a pseudo-isomorphism. Given this, the tautological short exact sequence
\[ 0 \to \iota_1(M)\cap N \xrightarrow{x \mapsto (x,x)} \iota_1(M)\oplus N \xrightarrow{(x,y) \mapsto x-y} \iota_1(M)+ N \to 0\]
implies that the composite map 
\[ M \oplus N \xrightarrow{(\iota_1,{\rm id})} \iota_1(M)\oplus N \xrightarrow{(x,y) \mapsto x-y} \bigoplus_{\tau\in \mathcal{T}} A/(a_\tau)\]
is a pseudo-isomorphism. This proves claim (i) with $L_\tau = (a_\tau)$ for each $\tau \in \mathcal{T}$.
    
In the remainder of the argument we no longer require, except when explicitly stated, that $M$ is a torsion module, but we do assume that the ring $Q(A)$ is semisimple, and hence regular. Then, since the localisation of $A'$ at each maximal ideal is a valuation ring, results of Endo \cite[\S5, Prop. 10, Prop. 11 and Cor.]{endo} imply that $A'$ is the direct product $\prod_{t \in T}A'_{t}$ over a finite index set $T$ of semi-hereditary (or Pr\"ufer) domains $A'_t$. 

In particular, if $M$ is an amenable torsion module, then $\cP(M)$ is finite and, for each $t\in T$, the ring $A_t'$ is a semi-local Pr\"ufer domain and the $A_t'$-component of $M'$ is both finitely presented and torsion. In this case, therefore, we can apply the stronger structure theorem of Fuchs and Salce \cite[Cor. III.6.6, Th. V.3.4]{fuchs-salce} to each ring $A_t'$ in order to deduce the existence of an isomorphism (\ref{warfield-structure}) for which the module $N'$ is zero. Then, in this case, the module ${\rm coker}(\iota_1)' = {\rm coker}(\psi)$ vanishes and so ${\rm coker}(\iota_1)_\fp$, and hence also $N_\fp$, vanishes for all $\fp$ in $\cP(M)$. 

Next we suppose, in addition, that every prime ideal in $\cP(M)$ is finitely generated and we claim this implies that every prime ideal of $A'$ is finitely generated. To see this we note every prime ideal of $A'$ is of the form $\mathfrak{B} = \mathfrak{B}_0\times \prod_{t \in T\setminus\{t_0\}}A'_t$ where $\mathfrak{B}_0$ is a prime ideal of the domain $A'_{t_0}$ for some $t_0\in T$. If $\mathfrak{B}_0 = (0)$, then $\mathfrak{B}$ is finitely generated. If $\mathfrak{B}_0 \not= (0)$, then $\mathfrak{Q} := (0)\times \prod_{t \in T\setminus\{t_0\}}A'_t$ is a prime ideal of $A'$ that is strictly contained in $\mathfrak{B}$. Now, since $\cP(M)$ is assumed to be finite, the prime avoidance lemma implies that $\mathfrak{B}$ and $\mathfrak{Q}$ correspond to prime ideals $\fp$ and $\fp_1$ of $A$ with $\fp_1 \subsetneq \fp \subseteq \fq$ for some $\fq \in \cP(M)$. In particular, since $\fq$ has height one, this implies  $\fp = \fq$ and hence that $\mathfrak{B}$ is finitely generated, as claimed.  

At this stage, we can apply Cohen's Theorem \cite[Th. 2]{cohen} to deduce that $A'$, and hence each component $A'_t$, is Noetherian. 
Then the localisation $A'_\mathfrak{B}$ of $A'$ at each prime ideal $\mathfrak{B}$ is Noetherian, a domain (since each component $A'_t$ of $A'$ is a domain) and either a field (if $\mathfrak{B}$ corresponds to the zero ideal of some component $A_t'$) or a valuation ring (by Remark \ref{prime-avoidance-remark} and the assumption $M$ is amenable) and hence therefore a discrete valuation ring. Taken together, these facts imply that every component ring $A'_t$ of $A'$ is a Dedekind domain. We can therefore appeal to the usual structure theorem for finitely generated torsion modules over such rings to deduce that each ideal $(a_\tau')$ in the decomposition (\ref{warfield-structure})  (with $N' = (0)$) can be replaced by the $S$-localisation of a power of a prime ideal in $\cP(M)$, and hence each ideal $(a_\tau)$ in (\ref{hom-commutes-with-localisation}) by a power of a prime ideal in $\cP(M)$. In particular, in this case the module $N$ is pseudo-null (since we already observed that $N_\fp$ vanishes for all $\fp$ in $\cP(M)$) and so can be taken to be zero, as required to prove claim (ii)(a). 

To prove claim (ii)(b), we do not assume either that $M$ is torsion or that $M_\tor$ is amenable. We do however continue to assume that $Q(A)$ is semisimple and hence, by the above argument, that $A'$ is a finite direct product of semi-hereditary domains. Thus, by the general result of \cite[\S5, Cor.]{endo}, we know that $M_\tf'$ is a projective $A'$-module and hence that there exists an isomorphism of $A'$-modules of the form $M' \cong M'_\tf \oplus M_\tor'$.
    
Now, since $M$ is a finitely presented $A$-module, the natural map 
    \begin{align*}
        \Hom_A(M,M_\tor)' \to \Hom_{A'}(M',M_\tor')
    \end{align*}
    is bijective. In particular, there exists a homomorphism $\phi: M \to M_\tor$ and an element $s_1 \in S$ with the property that $\phi'/s_1$ corresponds under this identification to the projector of $M'$ onto $M'_\tor$. As such, $\phi'/s_1$ restricts to the submodule $M_\tor'$ to give the identity. We can therefore find an element $s_2$ of $S$ such that the map $\tau := s_2\cdot\phi$ restricted to $M_\tor$ is equal to $s_1s_2\cdot \id_{M_\tor}$.
    
    We now write $\pi$ for the canonical projection $M \to M_\tf$ and consider the map 
    \[ \kappa : M \to M_\tf \oplus M_\tor; \quad m \mapsto (\pi(m), \tau(m)).\]
One then checks that $\ker(\kappa) = \ker(\tau)\cap M_\tor$ and that $\coker(\kappa)$ is equal to the cokernel of the endomorphism of $M_\tor$ induced by $\tau$ and, since $s_1s_2 \in S$,  these modules are both pseudo-null. It follows that  the above map $\kappa$ is the required pseudo-isomorphism.
\end{proof}

In view of Theorem \ref{bourbaki-theorem}(ii), the following class of rings will be of interest in the sequel.  

\begin{definition}\label{amenable ring def} A commutative unital ring $A$ will be said to be \textit{amenable} if it has both of the following properties: 
\begin{itemize}
\item[(P$_3$)] $Q(A)$ is semisimple. 
\item[(P$_4$)] Finitely presented torsion $A$-modules are amenable (in the sense of Definition \ref{tame-definition}). 
\end{itemize}
\end{definition}

It is clear that a Noetherian integrally closed domain (or equivalently, a Noetherian Krull domain) is amenable in the above sense and also such that every finitely generated module is finitely presented. For such rings, Theorem \ref{bourbaki-theorem} simply recovers the classical structure theorem of Bourbaki \cite[Chap. VII, \S\,4, Th. \@4 and Th. \@5]{bourbaki}. However, the latter result can also be applied in more general situations such as those listed below (that are relevant for later arguments).

\begin{examples}\label{tame-examples}\text{}  \

\noindent{}(i) Let $A$ be an arbitrary Krull domain. 
In this case $Q(A)$ is a field (and so semisimple), $\cP_A$ is non-empty,  
         the localisation of $A$ at each prime in $\cP_A$ is a discrete valuation ring and every non-zero ideal is contained in only finitely many prime ideals in $\cP_A$. In particular, if $M$ is a non-zero finitely generated torsion $A$-module, then $\cP_A(M)$ is finite (as it is the subset of $\cP_A$ comprising primes containing the annihilator of $M$) so that $M$ has property (P$_1$) and admits a pseudo-isomorphism of the form (\ref{pseudiosmorphism-display}) with $N = 0$. Further, if every prime ideal of $A$ of height one is finitely generated, then $A$ is amenable.\ 
         
\noindent{}(ii) If $A$ is a unique factorisation domain, 
then $A$ is a Krull domain for which every height one prime ideal is principal and so the above discussion implies $A$ is amenable. In fact, for such a ring, the only essential difference between the argument of Theorem \ref{bourbaki-theorem} and that of Bourbaki  referred to above is that we require the module $M$ to be finitely presented, rather than merely finitely generated, in order to guarantee the existence of the isomorphism (\ref{hom-commutes-with-localisation}). 

\noindent{}(iii) Let $R$ be a $p$-adically complete integrally closed domain, $G$ a finite abelian group and $A$ the group ring $R[G]$. Write $f$ for the canonical (integral) ring homomorphism $R \to A$, $f^\ast: {\rm Spec}(A) \to {\rm Spec}(R)$ for the induced morphism of spectra and $f^\ast(M)$ for each $A$-module $M$ for the $R$-module obtained by restriction through $f$. Note that $Q(A)$ is semisimple and  $f^\ast$ is surjective with finite fibres. Fix $\mathfrak{p}\in {\rm Spec}(A)$ and set $\mathfrak{q} := f^\ast(\mathfrak{p})$. Then $\mathfrak{p}$ is finitely generated (over both $A$ and $R$) if $\mathfrak{q}$ is finitely generated, and if $R_\mathfrak{q}$ is a valuation ring, then $A_\mathfrak{p}$ is a valuation ring for every $\mathfrak{p}\in (f^\ast)^{-1}(\mathfrak{q})$ if and only if $|G| \notin \mathfrak{q}$. In addition, for any finitely generated $A$-module $M$ the following equivalences are valid:  
\begin{itemize}
    \item[$\bullet$] $M$ is finitely presented (over $A$) $\Longleftrightarrow$  $f^\ast(M)$ is finitely presented (over $R$);
    \item[$\bullet$] $M$ is a torsion $A$-module $\Longleftrightarrow$  $f^\ast(M)$ is a torsion $R$-module; 
    \item[$\bullet$] $f^\ast(\cP_{\!A}(M)) = \cP_{\!R}(f^\ast(M))$ and so $\cP_{\!A}(M)$ is finite $\Longleftrightarrow$ $\cP_{\!R}(f^\ast(M))$ is finite;
    \item[$\bullet$] $M$ is a pseudo-null $A$-module $\Longleftrightarrow$ $f^\ast(M)$ is a pseudo-null $R$-module. 
\end{itemize}
In particular, if  $f^\ast(M)$ is a finitely presented amenable torsion $R$-module for which no prime ideal in $\cP_{\!R}(f^\ast(M))$ contains $|G|$, then $M$ is a finitely presented amenable torsion $A$-module. However, if $p$ divides $|G|$, then $A$ is not an amenable ring since $A_\fp$ is not a valuation ring for any height one prime ideal $\fp$ that contains $p$ (and each such $\fp$ is contained in the support of the finitely presented torsion $A$-module $M= A/Ap$).   

\noindent{}(iv) Fix a natural number $n$ and, in the setting of (iii), take $R$ to be the completed $p$-adic group ring $\ZZ_p[[\ZZ^n_p]]$ and assume $p$ divides $|G|$. Then $A = R[G]$ is Noetherian (but not integrally closed or a domain) and the above discussion implies that a finitely generated torsion $A$-module $M$ is amenable if its $\mu$-invariant (as an $R$-module) vanishes. Hence, in this case, Theorem \ref{bourbaki-theorem} provides the following `equivariant' refinement of the structure theorem for Iwasawa modules: if the $\mu$-invariant of $M$ vanishes, then $\cP_A(M)$ is finite and $M$ is pseudo-isomorphic, as an $A$-module, to a finite direct sum of modules of the form $A/\mathfrak{p}^{e_\mathfrak{p}}$, with $\fp\in \cP_A(M)$ and $e_\fp\in \mathbb{N}$. 
\end{examples}


\subsection{Generalised characteristic ideals}\label{char sec} 

If $Q(A)$ is semisimple, then one can define a generalised notion of `characteristic ideal' as follows. 

\begin{definition} Assume $Q(A)$ is semisimple and let $M$ be a finitely presented torsion amenable $A$-module. Then, by Theorem \ref{bourbaki-theorem}(ii)(a), for each $\fp$ in $\cP_{\!A}(M)$ there exists a finite set $\{e_{\fp,i}\}_{1\le i\le n_\fp}$ of natural numbers  $e_{\fp,i}$ for which there is a pseudo-isomorphism of $A$-modules 
\begin{equation}\label{pseudo choice} M \to \bigoplus_{\fp \in \cP_{\!A}(M)} \bigoplus_{i=1}^{i=n_\fp} A/\fp^{e_{\fp,i}}.\end{equation}
The \textit{characteristic ideal} of $M$ is the ideal of $A$ obtained by setting

\begin{align*}
    \mathrm{char}_A(M) := \prod_{\fp \in \cP_{\!A}(M)} \fp^{\sum_{i=1}^{i=n_\fp}e_{\fp,i}}.
\end{align*}
\end{definition}
\smallskip

For each natural number $n$, the completed group ring $\ZZ_p[[\ZZ_p^n]]$ is amenable (in the sense of Definition \ref{amenable ring def}), and the above ideals coincide with the classical characteristic ideals of finitely generated (torsion) $\ZZ_p[[\ZZ_p^n]]$-modules.

The next result shows that, in more general situations, these ideals still retain some of the key properties of classical characteristic ideals.

In the sequel we refer to a finitely presented $A$-module as `quadratically presented' if, for some natural number $d$, it lies in an exact sequence of the form 
\begin{equation}\label{quadratic} A^d \xrightarrow{\theta} A^d \xrightarrow{\pi} M \to 0.\end{equation} 
%

\begin{proposition}\label{char-lemma}
    If $Q(A)$ is semisimple, then the following claims are valid. 
        \begin{itemize}
    \item[(i)] If $M$ a finitely presented torsion $A$-module, then the following claims are valid. 
    \begin{itemize}
    \item[(a)] Assume $M$ is amenable. Then $\mathrm{char}_A(M)$ is independent of the choice of pseudo-isomorphism (\ref{pseudo choice}) and there exists a principal ideal $L$ of $A$ such that $L_\fq = \mathrm{char}_A(M)_\fq$ for all $\fq$ in $\cP_{\!A}$.  
    \item[(b)] Assume $A = R[G]$ for a $p$-adically complete Krull domain $R$ and a finite abelian group $G$. Then $M$ is amenable if every prime in $\cP_{\!R}(M)$ is finitely generated and doesn't contain $|G|$. If this is the case, then the ideal ${\rm char}_A(M)$ is principal and, if $M$ is quadratically presented, equal to the initial Fitting ideal ${\rm Fit}_A^0(M)$ of $M$ as an $A$-module.  
    \end{itemize}
        \item[(ii)] Let $0 \to M_1 \to M_2 \to M_3 \to 0$ be an exact sequence of finitely generated $A$-modules. Then the following claims are valid.
    \begin{itemize}
    \item[(a)] If $M_2$ is a finitely presented amenable torsion $A$-module, then $M_3$ is a finitely presented amenable torsion $A$-module and  $\mathrm{char}_A(M_2) \subseteq \mathrm{char}_A(M_3)$.
    \item[(b)] If $M_1$ and $M_3$ are finitely presented amenable torsion $A$-modules, then $M_2$ is a finitely presented amenable torsion $A$-module. and  
        \[ \mathrm{char}_A(M_2) = \mathrm{char}_A(M_1)\cdot\mathrm{char}_A(M_3).\] 
   \end{itemize}
  \end{itemize}
\end{proposition}

\begin{proof} The key point in the proof of claim (i)(a) is that, if $M$ is amenable, then for every $\fp$ in $\cP_A(M)$ the ring $A_\fp = A'_{\fp'}$ that occurs in the proof of  Theorem \ref{bourbaki-theorem}(ii)(a) is a discrete valuation ring. For each such $\fp$, therefore, the pseudo-isomorphism (\ref{pseudo choice}) implies that $\fp^{\sum_{i=1}^{i=n_\fp}e_{\fp,i}}A_\fp$ is equal to $\fp A_\fp$ raised to the power of the length of the $A_\fp$-module $M_\fp$, and hence implies the first assertion of (i)(a). For a similar reason, the existence of a pseudo-isomorphism  (\ref{pseudiosmorphism-display}) in which $N = 0$ (as $Q(A)$ is semisimple) implies that the second assertion of (i)(a) is true with $L$ taken to be the product $\prod_{\tau\in \mathcal{T}}L_\tau$. 

To prove claim (i)(b) we set $A := R[G]$  and write $f$ for the ring inclusion $R \to A$. Then the first assertion follows from the discussion in Examples \ref{tame-examples}(iii). In addition, in this case every prime ideal in $\cP_{\!A}(M)$ is principal and generated by a non-zero divisor since, for each $\fq$ in $\cP_{\!R}(M)$, the ring $R_\fq[G]$ is integrally closed (as $p \notin \fq$) and hence equal to a finite product of Dedekind domains that have only finitely many prime ideals (and so are principal ideal rings). Next we note that, for each $\fp$ in $\cP_{\!A}(M)$, the presentation (\ref{quadratic}) gives rise to an exact sequence of $A_\fp$-modules 
\begin{equation}\label{induced quadratic} A_\fp^d \xrightarrow{\theta_{\fp}} A_\fp^d \xrightarrow{\pi_\fp} M_\fp \to 0.\end{equation} 
In particular, since $M_\fp$ is a torsion module over the discrete valuation ring $A_\fp$, this sequence implies that, for any fixed pseudo-isomorphism (\ref{pseudo choice}), there are equalities    
\[ A_\fp\cdot \mathrm{det}(\theta) = \fp_\fp^{l_{\fp}( \mathrm{coker}(\theta_\fp))} = \fp_\fp^{l_{\fp}(M_\fp)} = \fp_\fp^{\sum_{i=1}^{i=n_\fp}e_{\fp,i}} = {\rm char}_A(M)_\fp. \]
Here we write $l_\fp(N)$ for the length of a finitely generated torsion $A_\fp$-module $N$ so that the second equality follows from (\ref{induced quadratic}) and the third from (\ref{pseudo choice}). 

These equalities imply that $(A\cdot \mathrm{det}(\theta))_\fq = {\rm char}_A(M)_\fq$ for all primes $\fq$ in $\cP_{\!R}$. In addition, the ideals ${\rm char}_{\!A}(M)$ and $A\cdot{\rm det}(\theta)$ are both free $A$-modules of rank one, and hence free $R$-modules of rank $|G|$ (in the first case this is true because, as observed earlier,  every prime in $\cP_A(M)$ is principal and generated by a non-zero divisor, and in the second case because, since $M$ is torsion, the exact sequence (\ref{quadratic}) implies ${\rm det}(\theta)$ is a non-zero divisor of $A$). One therefore has    
\begin{align*}\label{computation} {\rm Fit}_{\!A}^0(M) = A\cdot {\rm det}(\theta) = \bigcap_{\fq\in \cP_{\!R}} (A_\fq \cdot {\rm det}(\theta))_\fq =&\, \bigcap_{\fq\in \cP_{\!R}}{\rm char}_A(M)_\fq = \mathrm{char}_{A}(M),\notag\end{align*}
where the first equality follows directly from the definition of initial Fitting ideal (and (\ref{quadratic})) and the second and last equalities are true since $R$ is a Krull domain. This completes the proof of claim (i)(b). 
 
Turning to claim (ii), we note that the assertions regarding modules being torsion and finitely presented follow directly from the given exact sequence (and, in the latter case, the general result of \cite[Th. 2.1.2]{glaz}). In addition, for each prime ideal $\p$ of $A$, the given sequence induces a short exact sequence of $A_\p$-modules 
\[ 0 \to M_{1,\p} \to M_{2,\p} \to M_{3,\p} \to 0.\]
Assuming $M_2$ (or equivalently, both $M_1$ and $M_3$) to be torsion, these sequences imply an equality $\cP(M_2) = \cP(M_1) \cup \cP(M_3)$ that combines with Remark \ref{prime-avoidance-remark} to imply both of the assertions regarding amenability, and also combines with the observation made in the proof of claim (i)(a) to imply the stated inclusion, respectively equality, of characteristic ideals. 
\end{proof}

\subsection{Inverse limit rings} In this section we assume to be given an inverse system of rings 
\[ (A_n,\,\, \phi_n\!:\! A_n \to A_{n-1})_{n\in \mathbb{N}}\]
in which every homomorphism $\phi_n$ is surjective. We study the inverse limit ring 
\[ A := \varprojlim_n A_n.\]

For every $n$ we write $\phi_{\langle n\rangle}: A \to A_n$ for the induced projection map (so that $\phi_n\circ \phi_{\langle n\rangle} = \phi_{\langle n-1\rangle}$ for all $n$) and we use the decreasing separated filtration 
\[ I_\bullet := (I_n)_{n}\]
of $A$ that is obtained by setting $I_n := \ker(\phi_{\langle n\rangle})$ for every $n$. For an $A$-module $M$ and non-negative integer $n$ we define an $A_n$-module by setting   
\[ M_{(n)} := M/(I_n\cdot M) \cong (A/I_n)\otimes_A M \cong A_n\otimes_A M.\]
We say $M$ is `$I_\bullet$-complete' if the natural map 
\[ \mu_{M}: M \to \varprojlim_nM_{(n)}\]
is bijective, where the limit is with respect to the natural maps $\mu_{M,n}: M_{(n)}\to M_{(n-1)}$. 

\subsubsection{The general case} The following result records some useful general facts about the notion of $I_\bullet$-completeness. In this result we refer to the linear topology on $A$ induced by the subgroups $\{I_n\}_n$ as the `$I_\bullet$-topology'. 

\begin{lemma}\label{submodule-pro-residual} The following claims are valid for every $A$-module $M$. 

\begin{itemize}
\item[(i)] If $M$ is finitely generated, then $\mu_M$ is surjective but need not be injective.
\item[(ii)] $M$ is $I_\bullet$-complete if it is a finitely generated submodule of an $I_\bullet$-complete module. In particular, every  finitely generated ideal of $A$ is $I_\bullet$-complete. 
\item[(iii)] {Assume $M$ is $I_\bullet$-complete and there exists a natural number $a$ for which the $I_a$-adic topology on $A$ is finer than the $I_\bullet$-topology. Then $M$ is finitely generated if and only if $M_{(a)}$ is finitely generated.} 
\end{itemize}
\end{lemma}

\begin{proof} To prove claim (i) we fix a natural number $d$ for which there exists an exact sequence of $A$-modules of the form  
\begin{equation}\label{resolution} 0 \to K \xrightarrow{\subseteq} A^d \xrightarrow{\varphi} M \to 0.\end{equation}
For each $n$, we set $K'_n := \ker(\varphi_{(n)})$ and use the exact commutative diagram 
    
    \begin{equation*}\label{limit diagram 0}
        \xymatrix{ 
0 \ar[r] & K'_n \ar[r]^{\subseteq} \ar[d]^{\alpha_n} &A_n^{d} \ar[r]^{\varphi_{(n)}} \ar@{->>}[d]^{\phi_n^d} & M_{(n)} \ar[r] \ar@{->>}[d]^{\mu_{M,n}} &0\\
        0 \ar[r] &K'_{n-1} \ar[r]^{\subseteq} &A_{n-1}^{d} \ar[r]^{\varphi_{(n-1)}} &M_{(n-1)} \ar[r] &0.}
    \end{equation*}
Write $I_{[n]}$ for the image of $I_{n-1}$ in $A_n$. Then $\ker(\phi_n^d) = I_{[n]}^d$ and $\ker(\mu_{M,n}) = I_{[n]}\cdot M_{(n)}$. Thus, since each map $\phi_n^d$ is surjective, the Snake Lemma applies to the above diagram to imply that each map $\alpha_n$ is surjective. By passing to the limit over $n$ of these diagrams, and noting $A^d$ is $I_\bullet$-complete (as a direct consequence of the definition of $A$ as a limit), we thus obtain the bottom row of the exact commutative diagram

    \begin{equation}\label{limit-diagram}
        \begin{tikzcd}
            0 \arrow[r] &K \arrow[r] \arrow[d] &A^d \arrow[r, "\varphi"] \arrow[d, "\mathrm{id}"] &M \arrow[r] \arrow[d, "\mu_M"] &0\\
            0 \arrow[r] &\varprojlim_n K_n' \arrow[r] &A^d \arrow[r] \arrow[r] &\varprojlim_n M_{(n)} \arrow[r] &0.
        \end{tikzcd}
    \end{equation}
This diagram implies $\mu_M$ is surjective. To give an example in which $\mu_M$ is not injective we take $A_n$ to be the  power series ring $\ZZ_p[[X_1,\dots, X_n]]$ over $\ZZ_p$ in $n$ commuting indeterminates $X_i$ and $\phi_n$ to be the projection map $A_n\to A_{n-1}$ induced by sending $X_n$ to $0$. In this case $A$ identifies with one version (see \cite{chang}) of the power series ring over $\ZZ_p$ in a countable number of commuting indeterminates $\{X_i\}_{i \in \mathbb{N}}$. We then define $K$ to be the ideal of $A$ generated by the set $\{pX_1\} \cup \{ X_n-pX_{n+1}\}_{n \in \NN}$ and $M$ to be the quotient $A/K$. In this case one computes that $\varprojlim_n K_n' = A$ and so the above diagram (with $d=1$) implies $\mu_M$ is not injective. 

To prove the first assertion of claim (ii) we fix an injective map $\theta: M \to N$ in which $N$ is $I_\bullet$-complete. It is then enough to note that $\mu_M$ is injective as a consequence of the diagram 
    \begin{center}
        \begin{tikzcd}
            M \arrow[r, hookrightarrow, "\theta"] \arrow[d, "\mu_M"] &N \arrow[d, "\mu_N"]\\
            \varprojlim_n M_{(n)}  \arrow[r, "(\theta_{(n)})_n"] & \varprojlim_n N_{(n)} 
        \end{tikzcd}
    \end{center}
    and the fact that $\mu_N$ is injective. The second assertion of claim (ii) is then obvious. 
    
{The hypothesis on $a$ in claim (iii) implies that, for every $n \in \mathbb{N}$, one has $(I_a)^m \subseteq I_n$ for some $m \in \mathbb{N}$. Given this, and the fact $M$ is $I_\bullet$-complete, the stated result follows by a standard Nakayama's Lemma type argument (as in the proof of \cite[Th. 8.4]{matsumura}).} 
\end{proof} 

\subsubsection{The compact case}\label{compact section} In the sequel we say that the inverse limit $A$ is `compact' if each ring $A_n$ is endowed with a compact topology with respect to which the transition maps $\phi_n$ are continuous. In this case we endow $A$ with the corresponding inverse limit topology, so that $A$ is compact and, for every $n$, the ideal $I_n$ is closed and the projection map $\phi_{\langle n\rangle}$ is continuous.  

In particular, since $A$ is compact, the inverse limit functor is exact on the category of finitely generated $A$-modules and this fact allows us to prove a finer version of Lemma \ref{submodule-pro-residual}. 

Before stating the result, we note that if an $A$-module $N$ is pseudo-null, then the associated $A_n$-module $N_{(n)}$ need not even be torsion. Such issues mean that, in general, one cannot hope to compute the characteristic ideal of a finitely presented torsion $A$-module $M$ directly in terms of the associated $A_n$-modules $M_{(n)}$. 

Despite this difficulty, claim (iii) of the following result shows that such a reduction is possible for a natural family of compact rings $A$, at least after possibly replacing $M$ by a pseudo-isomorphic module. (In Proposition \ref{quadratic connection} below we will also prove a more concrete version of this result for certain power series rings.)

\begin{proposition}\label{compact prop} Assume $A$ is compact. Then the following claims are valid for any finitely presented $A$-module $M$. 
\begin{itemize}
\item[(i)] $M$ is $I_\bullet$-complete. 
\item[(ii)] If $M$ is an amenable torsion module, then $\mathrm{char}_{A}(M) =\varprojlim_n \phi_{\langle n\rangle }(\mathrm{char}_{A}(M))$, where the limit is taken with respect to the maps $\phi_n$. 
\item[(iii)] Assume that $A$, and also $A_n$ for every $n$, are unique factorisation domains and that $M$ is torsion. 
Then $M$ is pseudo-isomorphic to a finitely presented torsion $A$-module $\widetilde M$ that is $I_\bullet$-complete and such that  
\[ \mathrm{char}_{A}(M) = \varprojlim_n {\rm char}_{A_n}(\widetilde M_{(n)}),\]
where the limit is taken with respect to the maps $\phi_n$ and the $A_n$-modules $\widetilde M_{(n)}$ are torsion for all sufficiently large $n$.
\end{itemize}
\end{proposition}

\begin{proof} To prove claim (i) we fix an exact sequence of $A$-modules of the form (\ref{resolution}). Then, by assumption the $A$-module $K$ is finitely generated and thus, by Lemma \ref{submodule-pro-residual}(ii), $I_\bullet$-complete. Hence, by passing to the limit over $n$ of the induced exact sequences of (compact) $A_n$-modules $K_{(n)} \to A^d_n \to M_{(n)} \to 0$ one obtains an exact sequence of $A$-modules 
\[ 0 \to K \xrightarrow{\subseteq} A^d \to \varprojlim_n M_{(n)} \to 0.\]
Comparing this to (\ref{resolution}) one deduces the map $\mu_M$ is bijective, as required to prove claim (i). 

In the rest of the argument we assume that $M$ is torsion. Then, since ${\rm char}_A(M)$ is a finitely generated ideal of $A$ (cf. condition (P$_2$) in Definition \ref{tame-definition}), to prove claim (ii) it is enough to show that any finitely generated ideal $N$ of $A$ is equal to $\varprojlim_n\phi_{\langle n\rangle}(N)$, where the limit is taken with respect to the maps $\phi_n$. To see this, we note that the above argument (with $M = A/N$, $d = 1$ and $K = N$) implies that the map $\mu_{A/N}$ is bijective. The stated equality then follows from the corresponding exact commutative diagram (\ref{limit-diagram}) and the fact that, in this case, one has $K_n' = \phi_{\langle n\rangle}(N)$ for every $n$.   

To prove claim (iii) we recall (from Theorem \ref{bourbaki-theorem}(ii)(a)) that $M$ is pseudo-isomorphic as an $A$-module to a finite direct sum $\widetilde M := \bigoplus_{\tau \in \mathcal{T}}A/L_\tau$, where each $L_\tau$ is a principal ideal of $A$. In particular, since each $L_\tau$ is principal, $\widetilde M$ is finitely presented as an $A$-module and hence $I_\bullet$-complete by claim (i). In addition, one has    
\begin{equation}\label{key limit} {\rm char}_A(M) = {\rm char}_A(\widetilde M) = \prod_{\tau \in \mathcal{T}}L_\tau = \varprojlim_n \prod_{\tau \in \mathcal{T}}\phi_{\langle n\rangle}(L_\tau),\end{equation}
where the second equality follows from the argument of Proposition \ref{char-lemma}(i)(b) (with $G$ trivial) and the fact that all ideals in $\cP_{\!A}$ are principal (as $A$ is assumed to be a unique factorisation domain), and the last equality follows from claim (ii). 

Now, for each $n$, one has 
\[ \widetilde M_{(n)} = \bigoplus_{\tau \in \mathcal{T}}\bigl(A/L_\tau\bigr)_{(n)} \cong \bigoplus_{\tau \in \mathcal{T}}N_{\tau,n}\]
with $N_{\tau,n} := A_n/\phi_{\langle n\rangle}(L_\tau)$. These $A_n$-modules are finitely presented and, for any sufficiently large $n$, also torsion. In particular, since $A_n$ is a Krull domain, 
for any such $n$, and every $\mathfrak{q}$ in $\mathcal{P}_{\!A_n}$, one has 
\begin{align*} \bigl(\prod_{\tau \in \mathcal{T}}\phi_{\langle n\rangle}(L_\tau)\bigr)_\mathfrak{q} = \prod_{\tau \in \mathcal{T}}\phi_{\langle n\rangle}(L_\tau)_\mathfrak{q} 
=&\, (\mathfrak{q}\!\cdot\! A_{n,\mathfrak{q}})^{\sum_{\tau \in \mathcal{T}}l_\mathfrak{q}(N_{n,\tau})}\\
=&\, (\mathfrak{q}\!\cdot\! A_{n,\mathfrak{q}})^{l_\mathfrak{q}(\widetilde M_{(n)})}  = {\rm char}_{A_n}(\widetilde M_{(n)})_\mathfrak{q}.\end{align*} 
The principal $A_n$-ideals $\prod_{\tau \in \mathcal{T}}\phi_{\langle n\rangle}(L_\tau)$ and ${\rm char}_{A_n}(\widetilde M_{(n)})$ are therefore equal and so claim (iii) follows directly from (\ref{key limit}).  \end{proof}

\section{Weil-\'etale cohomology for curves over finite fields}

In this section we describe an application of the above results to the Iwasawa theory of curves over finite fields. 

We write $\mathcal{U}(G)$ for the set of subgroups of finite index of a profinite group $G$. 

\subsection{Galois groups and power series rings}\label{gg section} 
The Iwasawa algebra of $\ZZ_p^\NN$ over $\ZZ_p$ is the completed $p$-adic group ring 
    \begin{align*}\label{cgr}
        \ZZ_p[[\ZZ_p^\NN]] := \varprojlim_{U\in \mathcal{U}(\ZZ_p^\NN)} \ZZ_p[\ZZ_p^\NN/U],
    \end{align*}
    where the limit is taken respect to the natural projection maps. After fixing a $\ZZ_p$-basis $\{\gamma_i\}_{i\in \mathbb{N}}$ of $\ZZ_p^\NN$, the association $X_i\mapsto \gamma_i-1$ induces a (non-canonical) isomorphism of rings between $\ZZ_p[[\ZZ_p^\NN]]$ and the power series ring 
\begin{equation*}\label{psr notation} R := \varprojlim_n R_n \quad \text{with} \quad 
R_n := \ZZ_p[[X_1,\dots, X_n]]\end{equation*}
in commuting indeterminants $\{X_i\}_{i \in \NN}$. Here the inverse limit is taken with respect to the (surjective) $\ZZ_p$-linear ring homomorphisms 
\[ \rho_n: R_n \twoheadrightarrow R_{n-1}\]
that send $X_i$ to $X_i$ if $1\le i < n$ and to $0$ if $i = n$. For each $n$ we also use the maps 
\[ \iota_n : R_n \hookrightarrow R\quad \text{and}\quad \rho_{\langle n\rangle}: R \twoheadrightarrow R_n,\]
that are respectively the natural inclusion and the (surjective) $\ZZ_p$-linear ring homomorphism that sends $X_i$ to $X_i$ if $1\le i\le n$ and to $0$ if $i > n$ (so that the pair $(\iota_n,\rho_{\langle n\rangle})$ is a retract of rings and, for each $n>1$, one has $\rho_n\circ \rho_{\langle n\rangle} = \rho_{\langle n-1\rangle}$). 

We also fix a finite abelian group $G$ and consider the group rings 
\[ A := R[G] \quad \text{and}\quad A_n = R_n[G],\]
together with the maps $A_n \to A_{n-1}, A_n \to A$ and $A \to A_n$ that are respectively induced by $\rho_n, \iota_n$ and $\rho_{\langle n\rangle}$ (and which we continue to denote by the same notation). 

We then define a separated decreasing filtration $I_\bullet = (I_n)_n$ of $A$ by setting 
\[ I_n := \ker(\rho_{\langle n\rangle})\]
for each $n$, and we note that $A$ is $I_\bullet$-complete.  

Since the submodule of $I_n$ that is generated by $\{X_i\}_{i > n}$ is not finitely generated, the ring $A$ is not Noetherian (cf. Remark \ref{glaz rem} below) and its module theory is complicated. For instance, the example discussed in the proof of Lemma \ref{submodule-pro-residual}(i) shows that cyclic $A$-modules need not be $I_\bullet$-complete (or even pro-finite) and also, taking account of a result of Fujiwara et al \cite[Th. 4.2.2]{fujiwara-gabber-kato}, that $A$ does not have the weak Artin-Rees property relative to $p$. Nevertheless, claims (i) and (ii) of the following result ensure that the theory developed in \S\ref{structure section} can be applied in this setting.   

\begin{lemma}\label{A prelim result} For every $n$ the following claims are valid.
\begin{itemize}  
\item[(i)] The rings $R$ and $R_n$ are $p$-adically complete unique factorisation domains, and hence amenable (in the sense of \S\ref{char sec}). 
\item[(ii)] The ring $A$ is $p$-adically complete and compact (in the sense of \S\ref{compact section}) and is amenable if and only if $p$ does not divide $|G|$. 
%
\item[(iii)] If $\mathfrak{p}$ is a prime ideal of $A_n$, then $\iota_n(\mathfrak{p})A$ is a prime ideal of $A$. 
\end{itemize}
\end{lemma}

\begin{proof} The first assertion of claim (i) is classical in the case of $R_n$ and then follows from the general result of Nishimura \cite[Th. 1]{nishimura} in the case of $R$. Given the latter fact, the second assertion of claim (i) follows directly from Remark \ref{tame-examples}(iii). 

Next we note that, for each subgroup $U$ in $\mathcal{U}(\ZZ_p^\NN)$ the group ring $\ZZ_p[(\ZZ_p^\NN/U) \times G]$ is finitely generated over $\ZZ_p$ and hence compact with respect to the canonical $p$-adic topology. The (inverse limit) ring  $\ZZ_p[[\ZZ_p^\mathbb{N}\times G]]$ is therefore compact with respect to the induced inverse limit topology. This induces a compact topology on $A$ that is independent of the choice of  $\ZZ_p$-basis $\{\gamma_i\}_{i\in \mathbb{N}}$ and such that each ideal $I_n$ is closed. This proves the first two assertions of claim (ii) and then, since $R$ is amenable (by claim (i)), the final assertion  follows from the discussion in Examples \ref{tame-examples}(iii).   

To prove claim (iii) we note that $\mathfrak{P} := \iota_n(\p)A$ is a finitely generated ideal of $A$, and hence that the 
map $\mu_{A/\mathfrak{P}}$ is bijective by Proposition \ref{compact prop}(i). Since, for  $m>n$, the image of the natural map $\mathfrak{P}_{(m)} \to A_{(m)} = A_m$ is $\rho_{\langle m\rangle}(\mathfrak{P}) = \p[[X_{n+1},\dots , X_m]]$, these observations combine to give a composite ring isomorphism 
    \begin{align*}
        A/\mathfrak{P} \xrightarrow{\mu_{A/\mathfrak{P}}} \varprojlim_{m > n} \bigl(A/\mathfrak{P}\bigr)_{(m)} \cong \varprojlim_{m > n} A_m/\rho_{\langle m\rangle}(\mathfrak{P}) \cong \varprojlim_{m>n} (A_n/\p)[[X_{n+1}, \cdots, X_m]].
    \end{align*}
Hence, since each ring $(A_n/\p)[[X_{n+1}, \cdots, X_m]]$ is a domain, the limit is a domain and so $\mathfrak{P}$ is a prime ideal of $A$.\end{proof}

\begin{remark}\label{glaz rem} Since $R$ is a unique factorization domain, it is a finite conductor ring in the sense of Glaz \cite{glaz-fcd} (so that every ideal with at most two generators is finitely presented). However, as far as we are aware, it is not known whether  $R$ is a coherent ring. \end{remark}

\begin{remark} Every prime ideal of $R$ that is principal has height one (since if a generating element $x$ does not belong to any prime ideal in $\cP_{\!R}$, then $x^{-1}$ belongs to $R_\mathfrak{q}$ for every $\mathfrak{q}$ in $\cP_{\!R}$ and hence to $R = \bigcap_{\mathfrak{q}\in \cP_{\!R}}R_\mathfrak{q}$). Lemma \ref{A prelim result}(iii) therefore implies that $\iota_n(\mathfrak{p})R$ belongs to $\cP_{\!R}$ if $\fp$ belongs to $\cP_{\!R_n}$. This observation is a special case of a result of Gilmer \cite[Th. 3.2]{gilmour} and is also related to the second part of \cite[Prop. 2.3]{bbl}. \end{remark}

The following result proves a more concrete version of Proposition \ref{compact prop}(iii) in this case and shows that, for a natural class of torsion $A$-modules, our characteristic ideals coincide with the `pro-characteristic ideals' that are introduced by Bandini et al in \cite{bbl}.

\begin{proposition}\label{quadratic connection} Let $M$ be a quadratically presented amenable torsion $A$-module such that, for every $n$, the $A_n$-module $M_{(n)}$ is amenable. Then the following claims are valid. 

\begin{itemize}
\item[(i)] For any natural number $n$ for which the $A_n$-module 
    $M_{(n)}$ is torsion, the $A_n$-module $(M_{(n+1)})^{X_{n+1}=0}$ is pseudo-null. 
\item[(ii)] The pro-characteristic ideal (in the sense of \cite[Def. 1.3]{bbl}) of the $A$-module $\varprojlim_nM_{(n)}$ is equal to ${\rm char}_{\!A}(M)$. 
\end{itemize}
\end{proposition}

\begin{proof} The given hypotheses imply that the $A_{n+1}$-module $M_{(n+1)}$ is torsion and that $M_{(n+1)}$ and $M_{(n)}$ are both quadratically presented (over $A_{n+1}$ and $A_n$ respectively). Hence there are  equalities of $A_n$-ideals 

\begin{align*} {\rm char}_{A_n}\bigl((M_{(n+1)})^{X_{n+1}=0}\bigr)\cdot\rho_{n+1}\bigl({\rm char}_{A_{n+1}}(M_{(n+1)})\bigr) =&\, {\rm char}_{A_{n}}(M_{(n)})\\
                                           =&\, {\rm Fit}_{A_n}^0(M_{(n)})\\
                                           =&\, \rho_{n+1}\bigl({\rm Fit}_{A_{n+1}}^0(M_{(n+1)})\bigr) \\
                                           =&\, \rho_{n+1}\bigl({\rm char}_{A_{n+1}}(M_{(n+1)})\bigr). \end{align*}
Here the second and last equalities follow from Proposition \ref{char-lemma}(i)(b) (and Lemma \ref{A prelim result}(i)) and, given the identification $(M_{(n+1)})_{(n)} = M_{(n)}$, the first equality follows from the 
general result of \cite[Prop. 2.10]{bbl} (see also \cite[Lem. 4]{pr}) and the third from a standard property of Fitting ideals under scalar extension. In addition, since $M_{(n)}$ is a quadratically presented torsion 
$A_{n}$-module, the ideal ${\rm Fit}^0_{A_{n}}\bigl(M_{(n)}\bigr)$, and hence also, by the above equalities, $\rho_{n+1}\bigl({\rm char}_{A_{n+1}}(M_{(n+1)})\bigr)$, is principal and generated by a non-zero divisor. The above equalities therefore imply that ${\rm char}_{A_{n}}\bigl((M_{(n+1)})^{X_{n+1}=0}\bigr) = A_n$, and hence that 
$(M_{(n+1)})^{X_{n+1}=0}$ is a pseudo-null $A_{n}$-module, as claimed. 

In a similar way, for every $n$ Proposition \ref{char-lemma}(i)(b) implies that  
\[ {\rm char}_{A_{n}}(M_{(n)}) = {\rm Fit}_{A_n}^0(M_{(n)}) = \rho_{\langle n\rangle}\bigl({\rm Fit}_{A}^0(M)\bigr) = 
\rho_{\langle n\rangle }\bigl({\rm char}_{A}(M)\bigr).\]
Taking account of Proposition \ref{compact prop}(ii) (and Lemma \ref{A prelim result}(ii)), these equalities imply that the pro-characteristic ideal of the $A$-module 
$M\cong \varprojlim_nM_{(n)}$ is equal to ${\rm char}_{A}(M)$, as required. \end{proof}

\subsection{Structure results}\label{structure results section}

We henceforth fix a global function field $k$ of characteristic $p$ and a Galois extension $K$ of $k$ that is ramified at only finitely many places and such that the group $\Gamma := \Gal(K/k)$ is topologically isomorphic to a direct product $\ZZ_p^\NN\times G$ for a finite abelian group $G$. We fix such an isomorphism and, in addition, a finite non-empty set of places $\Sigma$ of $k$ that contains all places that ramify in $K$ but no place that splits completely in $K$. For every intermediate field $L$ of $K/k$ we set $\Gamma_{\!L} := \Gal(L/k)$ and, if $L/k$ is finite, we write $\mathcal{O}_L^\Sigma$  for the subring of $L$ comprising elements that are integral at all places outside those above $\Sigma$.  

\subsubsection{Statement of the main results} 

For a finite extension $F$ of $k$ in $K$, the result of \cite[Chap. V, Th.
1.2]{tb} implies that the sum 
\[ \theta_{F}^\Sigma := [F:k]^{-1}\sum_{\psi\in \Gamma_{\!F}^\ast}\sum_{\gamma\in \Gamma_{\!F}} \psi(\gamma^{-1})L_\Sigma(\psi,0) \]
is a well-defined element of $\ZZ_p[\Gamma_{\!F}]$, where $\Gamma_{\!F}^\ast$ denotes $\Hom(\Gamma_{\!F},\QQ_p^{c,\times})$ and $L_\Sigma(\psi,0)$ the value at $0$ of the $\Sigma$-truncated Dirichlet $L$-series of $\psi$ (here we use that, in the notation of loc. cit., $\theta_{F}^\Sigma$ is equal to
$\Theta_\Sigma(1)$ and, as $p =
{\rm char}(k)$, the integer $e$ is prime to $p$). Then the behaviour of Dirichlet $L$-series under inflation of characters implies that the elements $\theta_{F}^\Sigma$ are compatible with respect to the projection maps $\ZZ_p[\Gamma_{\!F'}] \to \ZZ_p[\Gamma_{\!F}]$ for $F \subset F'$ and so, for each extension $L$ of $k$ in $K$, we obtain a well-defined element of $\ZZ_p[[\Gamma_{\!L}]]$ by setting  

\[ \theta_{L}^\Sigma := \varprojlim_{U \in \mathcal{U}(\Gamma_{\!L})}\theta_{L^U}^\Sigma.\]
For each such  $L$ we also set 

\[ H^1( (\mathcal{O}_L^\Sigma)_{W{\rm\acute e t}},\ZZ_p(1)) := \varprojlim_{U\in \mathcal{U}(\Gamma_{\!L})}\bigl(\ZZ_p\otimes_\ZZ H^1( (\mathcal{O}_{L^U}^\Sigma)_{W{\rm\acute e t}},\mathbb{G}_m))\]
and both 
\[ \mathrm{Pic}^0(L)_p := \varprojlim_{U\in \mathcal{U}(\Gamma_{\!L})}(\ZZ_p\otimes_\ZZ\mathrm{Pic}^0(L^U))\quad\text{and}\quad \mathrm{Cl}(\mathcal{O}_L^\Sigma)_p := \varprojlim_{U\in \mathcal{U}(\Gamma_{\!L})}(\ZZ_p\otimes_\ZZ\mathrm{Cl}(\mathcal{O}_{L^U}^\Sigma)),\]
%
%
%
where $(-)_{W{\rm\acute e t}}$ denotes the Weil-\'etale site defined by Lichtenbaum in \cite[\S2]{lichtenbaum} and $\mathrm{Pic}^0(L^U)$ the degree zero divisor class group of $L^U$, and the respective limits are with respect to the natural corestriction and norm maps.

We fix a $\ZZ_p$-basis $\{\gamma_i\}_{i\in \mathbb{N}}$ of $\ZZ_p^\NN$ (as at the beginning of \S\ref{gg section}) and, for each $n\in \NN$, we write $\Gamma(n)$ for the $\ZZ_p$-module generated by $\{\gamma_i\}_{i > n}$ and $K_n$ for the fixed field of $\Gamma(n)$ in $K$ (so that $\Gamma_{\!K_n}$ is isomorphic to $\ZZ_p^n\times G$). We also write $\Gamma_{\!v}$ for 
the decomposition group in $\Gamma$ of each $v$ in $\Sigma$ and consider the following condition.   

\begin{hypothesis}\label{hyp} There exists a natural number $n_0$ such that, for every $v$ in $\Sigma$, the group $\Gamma(n_0) \cap \Gamma_{\!v}$ is not open in $\Gamma_{\!v}$. 
\end{hypothesis}

This hypothesis is satisfied in the setting of the main results of both Anglès et al \cite{abbl} and Bley and Popescu \cite{bp} and so the structural aspects of the next result complement these earlier results (see also Remarks \ref{abbl rem} and \ref{bp rem} below). 

We use the basis $\{\gamma_i\}_{i\in \mathbb{N}}$ of $\ZZ_p^\mathbb{N}$ to identify the completed $p$-adic group ring $\ZZ_p[[\Gamma]]$ with the group ring $A = R[G]$ of $G$ over the power series ring $R = \ZZ_p[[\ZZ_p^\mathbb{N}]]$. In the sequel we shall thereby regard  the inverse limit  
\[ M := H^1( (\mathcal{O}_K^\Sigma)_{W{\rm\acute e t}},\ZZ_p(1))\] 
as an $A$-module. For each $n$ we set $A_n:= R_n[G] \cong \ZZ_p[[\Gamma_{\!K_n}]]$ and $M_{(n)} := A_n\otimes_A M$. 

\begin{theorem}\label{main result} The $A$-module $M$ has the following properties.  

\begin{itemize}
\item[(i)] $M$ is finitely presented and, for every $n$, the $A_n$-module $M_{(n)}$ is isomorphic to $H^1( (\mathcal{O}_{K_{n}}^\Sigma)_{W{\rm\acute e t}},\ZZ_p(1))$. 
\end{itemize}
In the remainder of the result we assume that $K$ and $\Sigma$ satisfy Hypothesis \ref{hyp}.
\begin{itemize}
\item[(ii)] $M$ is torsion. 
\item[(iii)] If $|G|$ does not belong to any prime in $\cP_{\!A}(M)$, then there  exists a pseudo-isomorphism  
\[ M \to \bigoplus_{\fp \in \cP_{\!A}(M)}\bigoplus_{i=1}^{i=n_\fp} A/\fp^{e_{\fp,i}}\]
of $A$-modules, for which one has  
\[ \prod_{\fp \in \cP_{\!A}(M)}\fp^{\sum_{i=1}^{i=n_\fp}e_{\fp,i}} = A\cdot \theta_{K}^\Sigma.\]
\item[(iv)] If $|G|$ does not belong to any prime in either $\cP_{\!A}(M)$ or $\cP_{\!A_n}(M_{(n)})$ for any $n\ge n_0$, then for every such $n$ the $A_n$-modules 
\[   H^1( (\mathcal{O}_{K_{n+1}}^\Sigma)_{W{\rm\acute e t}},\ZZ_p(1))^{X_{n+1}=0} \quad \text{and} \quad{\rm Cl}(\mathcal{O}_{K_{n+1}}^\Sigma)_p^{X_{n+1} = 0} \]
are pseudo-null. 
\end{itemize}
\end{theorem}

This result has the following concrete consequence for the $A$-module $\mathrm{Pic}^0(K)_p$.  

\begin{corollary}\label{fg cor} Assume that $K$ and $\Sigma$ satisfy Hypothesis \ref{hyp}. Then $\mathrm{Pic}^0(K)_p$ is torsion over $R$. Further, if $\mathrm{Pic}^0(K)_p$ is finitely generated over $R$, then at most one place that ramifies in $K$ has an open decomposition subgroup and, if such a place $v$ exists, then $\Gamma_v = \Gamma$. \end{corollary}

{The proof of these results will occupy the remainder of \S\ref{structure results section}. 

\subsubsection{Preliminaries on Weil-\'etale cohomology} We first recall some general facts about Weil-\'etale cohomology. For this we write $\Der(\Lambda)$ for the derived category of complexes over a commutative Noetherian ring $\Lambda$. 

For a finite extension $F$ of $k$ in $K$ we also write $C_F$ for the unique
geometrically irreducible smooth projective curve with function
field $F$ and $j^\Sigma_{F}$ for the natural open immersion $\Spec(\mathcal{O}_{F}^\Sigma)
\to C_F$. We then define an object of $\Der(\ZZ_p[\Gamma_{\!F}])$ by setting
\[ D_{F,\Sigma}^\bullet := \R\mathrm{Hom}_{\ZZ_p}( \R\Gamma((C_F)_{{\rm 
\acute e t}},j^\Sigma_{F,!}(\ZZ_p)),\ZZ_p[-2]).\] 
%
%
We note that, for this object, there is a canonical composite isomorphism 
\begin{align}\label{weil-nw}H^1(D_{F,\Sigma}^\bullet) \cong&\, \ZZ_p\otimes_\ZZ H^1(\R\mathrm{Hom}_{\ZZ}( \R\Gamma((C_F)_{W{\rm
\acute e t}},j^\Sigma_{F,!}(\ZZ)),\ZZ[-2]))\\
\cong&\, \ZZ_p\otimes_\ZZ H^1( (\mathcal{O}_F^\Sigma)_{W{\rm\acute e t}},\mathbb{G}_m)\notag\\
=&\, H^1( (\mathcal{O}_F^\Sigma)_{W{\rm\acute e t}},\ZZ_p(1)),\notag\end{align}
where the first isomorphism is a consequence of \cite[Prop. 2.4(g)]{lichtenbaum} and the second of the duality theorem in Weil-\'etale cohomology \cite[Th. 5.4(a)]{lichtenbaum} and the equality follows directly from our definition of the last displayed module. 

We next recall (from the proof of \cite[Prop. 4.1]{db}) that $D_{F,\Sigma}^\bullet$ is acyclic in degrees greater than one and such that, for each intermediate field $F'$ of $F/k$, there is a projection formula isomorphism $\ZZ_p[\Gamma_{\!F'}]\otimes^{\Le}_{\ZZ_p[\Gamma_{\!F}]} D^\bullet_{F,\Sigma} \cong D^\bullet_{F',\Sigma}$ in $\Der(\ZZ_p[\Gamma_{\!F'}])$. These facts combine with (\ref{weil-nw}) to imply that the natural corestriction map $H^1( (\mathcal{O}_F^\Sigma)_{W{\rm\acute e t}},\mathbb{G}_m) \to H^1( (\mathcal{O}_{F'}^\Sigma)_{W{\rm\acute e t}},\mathbb{G}_m)$ induces an isomorphism of $\ZZ_p[\Gamma_{\!F'}]$-modules
\begin{equation}\label{codescent iso} \ZZ_p[\Gamma_{\!F'}]\otimes_{\ZZ_p[\Gamma_{\!F}]}
H^1( (\mathcal{O}_F^\Sigma)_{W{\rm\acute e t}},\ZZ_p(1)) \cong H^1((\mathcal{O}_{F'}^\Sigma)_{W{\rm\acute e t}},\ZZ_p(1)).\end{equation} 

\subsubsection{The proof of Theorem \ref{main result}} We fix an exhaustive separated decreasing filtration $(\Delta_n)_n$ of the subgroup $\ZZ_p^\mathbb{N}$ of $\Gamma$ by open subgroups. We set $F_n := L^{\Delta_n}$, write $J_n$ for the kernel of the natural projection map  
\[ A \twoheadrightarrow A_{[n]} :=  \ZZ_p[\Gamma_{\!F_n}] = \ZZ_p[\Gamma/\Delta_n] \cong \ZZ_p[ (\ZZ_p^\mathbb{N}/\Delta_n)][G],\]
and for any $A$-module $N$, respectively map of $A$-modules $\theta$, we set $N_{[n]} := A_{[n]}\otimes_AN$ and $\theta_{[n]} := A_{[n]}\otimes_A\theta$. Then 
\[ J_\bullet := (J_n)_n\]
is a separated decreasing filtration of $A$ with respect to which $A$ is complete. In addition, the isomorphisms (\ref{codescent iso}) with $F/F'$ equal to each $F_n/F_{n-1}$ imply the $A$-module $M$ is $J_\bullet$-complete and that, for every $n$, there is a natural isomorphism $M_{[n]} \cong H^1((\mathcal{O}_{F_n}^\Sigma)_{W{\rm\acute e t}},\ZZ_p(1))$.

Turning now to the proof of Theorem \ref{main result}, we first observe that the isomorphisms in the second assertion of claim (i) are directly induced by the descent isomorphisms (\ref{codescent iso}). We then claim that, to prove the finite-presentability of $M$ (and hence complete the proof of claim (i)), it suffices to inductively construct, for every $n$, an exact commutative diagram of $A_{[n]}$-modules 
\begin{equation}\label{construct diagram} \begin{CD}
A_{[n]}^d @> \theta_n >> A_{[n]}^d @> \pi_n >> M_{[n]} @> >> 0\\
@V\tau_n^0 VV @V\tau_{n}^1 VV @V \tau_{n} VV \\
A_{[n-1]}^d @> \theta_{n-1} >> A_{[n-1]}^d @> \pi_{n-1} >> M_{[n-1]} @> >> 0\end{CD}\end{equation}
in which the natural number $d$ is independent of $n$, all maps $\pi_n$ and $\tau_n^0$ are surjective and $\tau^1_n$ and $\tau_n$ are the tautological projections. To justify this reduction we use the fact that $\Delta_{n-1}/\Delta_n$ is a finite $p$-group and hence that the kernel of the projection $A_{[n]} \to A_{[n-1]}$ is contained in the Jacobson radical of (the finitely generated $\ZZ_p$-algebra) $A_{[n]}$. This in turn implies that the natural maps $\mathrm{GL}_d(A_{[n]}) \to \mathrm{GL}_d(A_{[n-1]})$ are surjective and hence, since $A$ is $J_\bullet$-complete, that the inverse limit of $A_{[n]}^d$ with respect to the maps $\tau_n^0$ is isomorphic to $A^d$. Then, since $M$ is also $J_\bullet$-complete (and the inverse limit functor is exact on the category of finitely generated $\ZZ_p$-modules), by passing to the limit over $n$ of the above diagrams one obtains an exact sequence of $A$-modules 
\begin{equation}\label{fp} A^d \xrightarrow{\theta} A^d \xrightarrow{\pi} M \to 0\end{equation}
(in which $\theta = \varprojlim_n\theta_n$ and $\pi = \varprojlim_n\pi_n$) which shows directly that $M$ is a finitely-presented $A$-module. 

To complete the proof of claim (i), we must therefore construct the diagrams (\ref{construct diagram}). To do this, we note that $F_1$ is a finite extension of $k$ and hence that the $A_{[1]}$-module $M_{[1]}$ is finitely generated. We can therefore fix a natural number $d$ and a subset $\{m_i\}_{1\le i\le d}$ of $M$ whose image in $M_{[1]}$ generates $M_{[1]}$ over $A_{[1]}$. For each $n$, we write $m_{i,n}$ for the projection of $m_i$ to $M_{[n]}$. 
 We then note that, just as above, the kernel of the projection $A_{[n]} \to A_{[1]}$ lies in the Jacobson radical of $A_{[n]}$, and hence that the  tautological isomorphism $A_{[1]}\otimes_{A_{[n]}} M_{[n]} \cong M_{[1]}$ combines with Nakayama's Lemma (for the category of $A_{[n]}$-modules) and our choice of elements $\{m_i\}_{1\le i\le d}$ to imply $\{m_{i,n}\}_{1\le i\le d}$ generates the $A_{[n]}$-module $M_{[n]}$. We therefore obtain the right hand commutative square in (\ref{construct diagram}) by defining $\pi_n$ (and similarly $\pi_{n-1}$) to be the map of $A_{[n]}$-modules that sends the $i$-th element in the  standard basis of $A_{[n]}^d$ to $m_{i,n}$. 

We next recall from the proof of \cite[Prop. 4.1]{db} that $D^\bullet_{F_n,\Sigma}$ can be represented by a complex $P_n \xrightarrow{\theta_n} A_{[n]}^d$ in which $P_n$ is a finitely generated projective $A_{[n]}$-module (placed in degree zero), $\im(\theta_n) = \ker(\pi_{n})$ and 
$\pi_{n}$ induces an isomorphism between $\mathrm{coker}(\theta_n)$ and $M_{[n]}$. Then, since $A_{[n]}$ is a finite product of local rings and the $A_{[n]}$-equivariant Euler characteristic of $D^\bullet_{F_n,\Sigma}$ vanishes (by Flach \cite[Th. 5.1]{flach}), it follows that the $A_{[n]}$-module $P_n$ is free of rank $d$ (and so, after changing $\theta_n$ if necessary, can be taken to be $A_{[n]}^d$). In particular, if we choose both of the rows in (\ref{construct diagram}) in this way, then they are exact and so the commutativity of the right hand square reduces us to proving the existence of a map $\tau_n^0$ that makes the left hand square commute and is also surjective. To do this we can first choose a morphism of $A_{[n-1]}$-modules $\tau'_n:(A_{[n]}^d)_{[n-1]} \to A_{[n-1]}^d$ for which the associated diagram 

\[ \begin{CD} (A_{[n]}^d)_{[n-1]} @> (\theta_n)_{[n-1]} >> (A_{[n]}^d)_{[n-1]}\\
@V\tau'_n VV @V \cong V (\tau_{n}^1)_{[n-1]} V\\
A_{[n-1]}^d @> \theta_{n-1} >> A_{[n-1]}^d\end{CD}\]
commutes and represents the canonical isomorphism $A_{[n-1]}\otimes^{\Le}_{A_{[n]}} D^\bullet_{F_n,\Sigma} \cong D^\bullet_{F_{n-1},\Sigma}$. In particular, since the morphism of complexes represented by this diagram is a quasi-isomorphism and $(\tau_n^1)_{[n-1]}$ is bijective, the map $\tau'_n$ must also be bijective. The composite map 
\[ \tau_n^0: A_{[n]}^d \to (A_{[n]}^d)_{[n-1]} \xrightarrow{\tau'_n} A_{[n-1]}^d\]
is then surjective and such that the diagram (\ref{construct diagram}) commutes, as required to complete the proof of claim (i).  }

In the rest of the argument we assume that $K$ and $\Sigma$ satisfy Hypothesis \ref{hyp}. 

To prove claim (ii) we note that the inclusion $R \to A$ is integral and hence that $M$ is a torsion over $R$ if and only if it is torsion over $A$. The exact sequence (\ref{fp}) therefore implies that $M$ is torsion over $R$ if and only if ${\rm det}(\theta)$ is a non-zero divisor of $A$. To investigate this condition, we recall that, for each $n$,  $K_n$ denotes $K^{\Gamma(n)}$ and we set $\Gamma_{\!n} := \Gamma/\Gamma(n) = \Gal(K_n/k)$ so that $A_n = \ZZ_p[[\Gamma_n]]$. We also write  $I_\bullet := (I_n)_n$ for the decreasing filtration of $A$ in which each $I_n$ is the kernel of the projection map $\pi_n: A \to A_n$. 

Then, for every $n\ge n_0$, Hypothesis \ref{hyp} implies that the decomposition subgroup in $\Gamma_{\!n}$ of every place in $\Sigma$ is infinite. Hence, for each such $n$, the results of \cite[Prop. 4.1 and Prop. 4.4]{db} combine to imply that $\pi_n(\mathrm{det}(\theta))$ and $\theta_{K_n}^\Sigma$ are non-zero divisors of $A_n$ such that 
\begin{equation}\label{n equality} A_n\cdot \pi_n({\rm det}(\theta)) = A_n\cdot \theta_{K_n}^\Sigma.\end{equation}
This implies, in particular, that ${\rm det}(\theta) = (\pi_n({\rm det}(\theta)))_{n\ge n_0}$ is a non-zero divisor in the ring $A = 
\varprojlim_n A_n = \varprojlim_{n \ge n_0}A_n$, and so claim (ii) is proved. 

To prove claim (iii), we note that the stated hypotheses imply that the $A$-module $M$ is finitely presented (by claim (i)), torsion (by claim (ii)) and amenable (by Proposition \ref{char-lemma}(i)(b) and Lemma \ref{A prelim result}(i)), and hence that Theorem \ref{bourbaki-theorem}(ii)(a) implies the existence of a pseudo-isomorphism of $A$-modules 
\begin{equation}\label{fixed pi} M \to \bigoplus_{\fp \in \cP_{\!A}(M)}\bigoplus_{i=1}^{i=n_\fp} A/\fp^{e_{\fp,i}}.\end{equation}
%

Next we note that, as $\pi_n({\rm det}(\theta))$ is a non-zero divisor for each $n\ge n_0$, the equality (\ref{n equality}) implies the existence for each such $n$ of an element $u_n$ of $A_n^\times$ with $\pi_n({\rm det}(\theta)) = u_n\cdot \theta_{K_n}^\Sigma$. In particular, since each $\theta_{K_n}^\Sigma$ is a non-zero divisor, the family $u := (u_n)_{n\ge n_0}$ belongs to $A^\times = 
 \varprojlim_{n \ge n_0}A_n^\times$ and is such that ${\rm det}(\theta) = u\cdot \theta_\Sigma^K$. One therefore has   
\[ \prod_{\fp \in \cP_{\!A}(M)}\fp^{\sum_{i=1}^{i=n_\fp}e_{\fp,i}} = {\rm char}_{\!A}(M) = {\rm Fit}_A^0(M) = 
A\cdot {\rm det}(\theta) = A\cdot\theta_\Sigma^K,\]
where the first equality follows from the pseudo-isomorphism (\ref{fixed pi}) (and the definition of characteristic ideals), the second from Proposition \ref{char-lemma}(i)(b) and the third from the resolution (\ref{fp}).  This proves claim (iii).  


Turning to claim (iv) we note that the resolution (\ref{fp}) combines with the isomorphisms in claim (i) to imply that, for each $n$, the $A_n$-module ${\rm cok}(A_n\otimes_A\theta) \cong  A_n\otimes_AM = M_{(n)}$ is isomorphic to $H^1( (\mathcal{O}_{K_{n}}^\Sigma)_{W{\rm\acute e t}},\ZZ_p(1))$. 

In particular, if $n \ge n_0$, then this module is torsion since ${\rm det}(A_n\otimes_A\theta) = \pi_n(\mathrm{det}(\theta))$ is a non-zero divisor. Hence, under the stated hypothesis on primes in $\cP_{\!A}(M)$ and $\cP_{\!A_n}(M_{(n)})$, the first assertion of claim (iv) follows directly from the argument of Proposition \ref{quadratic connection}. 

The second assertion of claim (iv) is then true since, after taking account of the isomorphisms (\ref{weil-nw}), the $A_n$-module 
$\mathrm{Cl}(\mathcal{O}_{K_n}^\Sigma)_p$ identifies with a submodule of $H^1( (\mathcal{O}_{K_{n}}^\Sigma)_{W{\rm\acute e t}},\ZZ_p(1))$ (cf. \cite[(4)]{db}).

\subsubsection{The proof of Corollary \ref{fg cor}} For each subset $\Sigma'$ of $\Sigma$ we write $\epsilon_{\Sigma'}$ for the canonical  projection map $\bigoplus_{v \in \Sigma'} \ZZ_p[[\Gamma/\Gamma_{\!v}]] \to 
\ZZ_p$. Then 
the exact sequence of \cite[(4)]{db} induces an exact sequence of $A$-modules 

\begin{equation}\label{last ses} 0 \to \mathrm{Cl}(\mathcal{O}_{K}^\Sigma)_p \to M \to \ker(\epsilon_\Sigma) \to 0,\end{equation}
and the exact sequences \cite[(5) and (6)]{db} combine to give an exact sequence of $A$-modules 
\begin{equation}\label{last ses2}  \ker(\epsilon_{\Sigma^K_{\mathrm{fin}}}) \to \mathrm{Pic}^0(K)_p \to \mathrm{Cl}(\mathcal{O}_{K}^\Sigma)_p \to \ZZ_p/(n_K) \to 0,\end{equation}
in which $\Sigma^K_\mathrm{fin}$ is the subset of $\Sigma$ comprising places that have finite residue degree in $K/k$ and $n_K$ is a (possibly zero) integer.  

We now assume Hypothesis \ref{hyp} is satisfied. In this case the $A$-module $M$ is finitely presented and torsion (by Theorem \ref{main result}(i) and (ii)) and the $A$-module $\ker(\epsilon_{\Sigma^K_{\mathrm{fin}}})$ is torsion. The first of these facts combines with the sequence (\ref{last ses}) to imply both that $\mathrm{Cl}(\mathcal{O}_{K}^\Sigma)_p$ is torsion and also (by using the general results of \cite[Th. 2.1.2, (2) and (3)]{glaz}) that $\mathrm{Cl}(\mathcal{O}_{K}^\Sigma)_p$ is finitely generated if and only if $\ker(\epsilon_\Sigma)$ is finitely presented. From the sequence (\ref{last ses2}) we can then deduce that $\mathrm{Pic}^0(K)_p$ is a torsion 
$A$-module (and hence a torsion $R$-module) and also that $\mathrm{Cl}(\mathcal{O}_{K}^\Sigma)_p$ is finitely generated  if 
 $\mathrm{Pic}^0(K)_p$ is finitely generated.   

To complete the proof we now argue by contradiction and, for this, the above observations imply it is enough to assume both that $\ker(\epsilon_\Sigma)$ is finitely presented (over $A$) and that there are either two places $v_1$ and $v_2$ in $\Sigma$ such that $\Gamma_{\!v_1}$ and $\Gamma_{\!v_2}$ are open, or at least one place $v_1$ in $\Sigma$ for which $\Gamma_{\!v_1}$ is open and not equal to $\Gamma$. We then define an open subgroup of $\Gamma$ by setting $\Gamma':= \Gamma_{\!v_1}\cap \Gamma_{\!v_2}$ in the first case and $\Gamma' := \Gamma_{\!v_1}$ in the second case, we set $A' := \ZZ_p[[\Gamma']]$ and we write $I$ and $I'$ for the kernels of the respective projection maps $A \to \ZZ_p$ and $A' \to \ZZ_p$. 

Then the definition of $\Gamma'$ ensures that the $A'$-module $\ker(\epsilon_\Sigma)$ is both finitely-presented and contains a direct summand that is isomorphic to the trivial module $\ZZ_p$. This implies (via \cite[Th. 2.1.2(4)]{glaz}) that $\ZZ_p$ is finitely-presented as an $A'$-module and hence, by applying \cite[Lem. 2.1.1]{glaz} to the short exact sequence 
\[ 0 \to I'\to A' \to \ZZ_p \to 0,\]
that $I'$ is finitely generated over $A'$. However, writing $d$ for the order of $\Gamma/\Gamma'$, there exists an exact sequence of $A'$-modules 
\[0 \to (I')^d \to I \to \ZZ_p^{d-1}\]
and this implies $I$ is finitely generated over $A'$, and hence over $A$, and this is a contradiction. This proves Corollary \ref{fg cor}. 

\begin{remark}\label{abbl rem} For the Carlitz-Hayes cyclotomic extensions $K/k$ considered by Anglès et al in \cite{abbl}, one has $\Gamma = \ZZ_p^\mathbb{N}$ (so $A = R$) and $\Sigma = \{v\}$ with $v$ a place that is totally ramified in $K$. In this case one has $\Gamma_{\!v} = \Gamma$ and it can also be checked the integer $n_K$ in (\ref{last ses2}) is not divisible by $p$, and so the exact sequences (\ref{last ses}) and (\ref{last ses2}) combine to induce identifications 
\[ M = \mathrm{Cl}(\mathcal{O}_{K}^\Sigma)_p = {\rm Pic}^0(K)_p.\]
This fact combines with Proposition \ref{quadratic connection}(ii) to imply that claims (iii) and (iv) of Theorem \ref{main result}   strengthen the main result of \cite{abbl}. \end{remark}

\begin{remark}\label{bp rem} Under the given hypotheses, the argument of Theorem \ref{main result} combines with Proposition \ref{char-lemma}(i)(b) to imply that the Fitting ideal $\mathrm{Fit}_{\!A}^0(M)$ is principal and equal to the generalised characteristic ideal $\mathrm{char}_{\!A}(M)$. To discuss a specific example, we take $K$ to be a Drinfeld modular tower extension $L_\infty$ of $k$ of the form specified by Bley and Popescu in \cite[\S2.2]{bp}, so that $A = R[G]$ with $G$ isomorphic to $\Gal(H_{\mathfrak{f}\fp}/k)$ for a `real' ray class field  $H_{\mathfrak{f}\fp}$ of $k$ relative to a fixed prime ideal $\fp$ and integral ideal $\mathfrak{f}$. Then by comparing (\ref{last ses}) to the exact sequences \cite[(24), (25), (26)]{bp}, and recalling that, under the hypotheses of Theorem \ref{main result}(iii), the localisation of $A$ at each prime ideal in $\cP_{\!A}(M)$ is a discrete valuation ring, one verifies an equality of principal ideals  
\[ \mathrm{Fit}^0_{\!A}(M) = {\rm Fit}^0_{\!A}(T_p(M^{(\infty)}_\Sigma)_\Gamma),\]
where the $A$-module $T_p(M_\Sigma^{(\infty)})_\Gamma = \varprojlim_n T_p(M_\Sigma^{(n)})_\Gamma$ is (quadratically presented and) defined in \cite[\S3.3]{bp} as an inverse limit over the $p$-adic Tate modules of  a canonical family of Picard $1$-motives. This connection shows that claims (iii) and (iv) of Theorem \ref{main result} strengthen \cite[Th. 1.3]{bp} (with $S = \Sigma$). Further, if $\fp$ decomposes in the field $H_{\mathfrak{f}\fp}$, then Corollary \ref{fg cor} implies that ${\rm Pic}^0(L_\infty)_p$ is not finitely generated as an $R$-module.\end{remark}


\end{document}